\documentclass{amsart}
\usepackage{graphicx}
\usepackage{graphicx,esint}
\usepackage{amssymb,amscd,amsthm,amsxtra}
\usepackage{latexsym}
\usepackage{epsfig}

\newtheorem{thm}{Theorem}[section]
\newtheorem{cor}[thm]{Corollary}
\newtheorem{lem}[thm]{Lemma}

\theoremstyle{definition}
\newtheorem{defn}[thm]{Definition}
\theoremstyle{remark}

\numberwithin{equation}{section}

\newcommand{\R}{\mathbb R}
\newcommand{\be}{\begin{equation}}
\newcommand{\ee}{\end{equation}}

\newcommand{\ep}{\eps}

\newcommand{\eps}{\varepsilon}

\newcommand{\p}{\partial}

\newcommand{\comment}[1]{}

\begin{document}

\title{Free boundary regularity for a problem with right hand side}%
\author{D. De Silva}
\address{Department of Mathematics, Barnard College, Columbia University, New York, NY 10027}
\email{\tt  desilva@math.columbia.edu}
\begin{abstract} We consider a one-phase free boundary problem with variable coefficients and non-zero right hand side. We prove that flat free boundaries are $C^{1,\alpha}$ using a different approach than the classical supconvolution method of Caffarelli. We use this result to obtain that Lipschitz free boundaries are $C^{1,\alpha}$.
\end{abstract}
\maketitle

\section{Introduction}

 Consider the following one-phase free boundary problem with variable coefficients and non-zero right hand side,
\begin{equation}\label{fb} \left \{
\begin{array}{ll}
    \sum_{i,j} a_{ij}(x) u_{ij} = f,   & \hbox{in $\Omega^+(u):= \{x \in \Omega : u(x)>0\}$,} \\
\ \\
    |\nabla u|= g, & \hbox{on $F(u):= \partial \Omega^+(u) \cap \Omega,$} \\
\end{array}\right.
\end{equation} with $\Omega$ a bounded domain in $\R^n,$  the coefficients $a_{ij} \in C^{0,\beta}(\Omega),$ $f \in C(\Omega) \cap
L^\infty(\Omega)$, and $g \in C^{0,\beta}(\Omega)$, $g \geq 0.$

In this paper we are concerned with the regularity of the set $F(u)$, that is the so-called free boundary of $u$. There is an extensive
literature on the regularity of the free boundary for this type of problem when $f
\equiv 0$. In the case of the Laplace operator, Caffarelli proved in his pioneer work \cite{C1} that Lipschitz free boundaries are $C^{1,\alpha}$, while in \cite{C2} he showed that ``flat" free boundaries are Lipschitz. The key step of the method in [C1,C2] consists in finding a family of comparison subsolutions using supconvolutions on balls of variable radii.

Higher regularity of the free boundary follows from the classical work of Kinderlehrer and  Nirenberg \cite{KN}.

Regularity results in the spirit of \cite{C1,C2} have been subsequently proved for more general operators. In \cite{W1,W2} Wang considered
concave fully nonlinear uniformly elliptic operators of the form $F(D^2u)$. The work
\cite{C1} was extended by Feldman \cite{F1,F2} to a class on nonconcave fully nonlinear uniformly elliptic operators of the type $F(D^2u, Du)$ and to certain nonisotropic problems. For operators with variable coefficients regularity results are proved in the work of Cerruti, Ferrari,
Salsa \cite{CFS}, and Ferrari, Salsa \cite{FS1, FS2}. Also, Ferrari and then Argiolas, Ferrari in \cite{Fe1, AF} considered a class of fully nonlinear operators of the form $F(D^2u, x)$ with H\"older dependence on $x$.

The results cited above follow the guidelines of \cite{C1,C2}. One purpose of this paper is to provide a different method to obtain that flat free boundaries are $C^{1,\alpha}$. The approach we use is quite flexible since it easily applies to more general nonlinear operators, even degenerate ones, and it also applies to two-phase problems.

In particular, when dealing with operators with variable coefficients we easily obtain that Lipschitz free boundaries are $C^{1,\alpha}.$ In fact our flatness result allows us to use a blow-up argument and reduce the problem to the case of constant coefficients operators.
Our strategy is largely inspired by the work of Savin \cite{S}.

We now state our main results (for the precise
definition of viscosity solutions we refer the reader to Section
2.) We  assume that the matrix $A=(a_{ij}(x))$ is positive definite.


\begin{thm}[Flatness implies $C^{1,\alpha}$]\label{flatmain} Let $u$ be a viscosity solution to \eqref{fb} in
$B_1$. Assume that
$0\in F(u)$, $g(0)=1$ and $a_{ij}(0)=\delta_{ij}$. There exists a universal
constant $\bar\ep>0$ such that, if the graph of $u$ is $\bar\ep$-flat
in $B_1$,i.e. \be\label{hyp1}(x_n-\bar \ep)^+ \leq u(x) \leq
(x_n+\bar \ep)^+, \quad x \in B_1\ee and
\be\label{hypflatmain}[a_{ij}]_{C^{0,\beta}(B_1)} \leq \bar\ep, \quad \|f\|_{L^\infty(B_1)} \leq \bar\ep, \quad [g]_{C^{0,\beta}(B_1)} \leq
\bar \ep,\ee then $F(u)$ is $C^{1,\alpha}$ in $B_{1/2}$.
\end{thm}

\begin{thm}[Lipschitz implies $C^{1,\alpha}$]\label{Lipmain} Let $u$ be a viscosity solution to \eqref{fb}.
 Assume that $0\in F(u)$ and $g(0)>0$. If $F(u)$ is a Lipschitz graph in
a neighborhood of $0$, then $F(u)$ is $C^{1,\alpha}$ in a
(smaller) neighborhood of $0$.
\end{thm}

In the theorem above, the size of the neighborhood where $F(u)$ is
$C^{1,\alpha}$ depends on the radius $\rho$ of the ball $B_\rho$
where $F(u)$ is Lipschitz, on the Lipschitz norm of $F(u)$, on
$[a_{ij}]_{C^{0,\beta}(B_\rho)}, \|g\|_{C^{0,\beta}(B_\rho)}$, and  $\|f\|_{L^\infty(B_\rho)}.$

We remark that the assumptions on the coefficients $a_{ij}(x)$ in Theorem \ref{flatmain} can be weakened to a Cordes-Nirenberg type condition:
$$\|a_{ij}-\delta_{ij}\|_{L^\infty(B_1)} \leq \delta(n).$$

As already pointed out, our strategy of the proof of Theorem \ref{flatmain} is inspired by \cite{S}. The main idea is to show that the graph of
$u$ enjoys an ``improvement of flatness" property, that is if the
graph of $u$ oscillates $\ep$ away from a hyperplane in $B_1$, then
in $B_{r_0}$ it oscillates $\ep r_0/2$ away from possibly a
different hyperplane. The key tool in proving this property will be a
Harnack type inequality for solutions to a one-phase free
boundary problem.

The proof of Theorem \ref{Lipmain} will follow via a blow-up
argument from Theorem \ref{flatmain} and the classical
theory in \cite{C1}.

The problem \eqref{fb}, in which a right hand side appears, is not specifically dealt with in any of the previous cited works.
Our interest in this problem arises in connection with the question of the regularity of the free surface which occurs in the classical
hydrodynamical problem for traveling two-dimensional gravity
water-waves with vorticity.
There has been considerable interest
in this problem in recent years, starting with the systematic
study of Constantin and Strauss \cite{CS}.

The physical situation is the following: a traveling wave of an
incompressible, inviscid, heavy fluid moves with constant speed
over an horizontal surface. Since the fluid is incompressible, the
flow can be described by a stream function $u$ which solves the
following free boundary problem (in 2D)
\begin{align*}
&\Delta u = - \gamma(u),  \quad \mbox{in $\Omega:=\{(x,y) \in \R^2 : 0<u(x,y)<B\}$}\\
&u=B,  \quad \mbox{on $y=0$}\\
&|\nabla u|^2+ 2gy=Q,  \quad \mbox{on $S:=\{u=0\}$,}
\end{align*} with $B, g$ fixed constants,$\gamma$ a given vorticity function
and $Q$ a parameter. Of special interest are those free boundaries which are given by the graph of a function $y=\psi(x)$. In the regions where $\psi$ is monotone decreasing (resp. increasing) the free boundary is Lipschitz with respect to the direction $e_1+e_2$ (resp. $e_2-e_1$) and moreover $Q-2gy>0$. As a consequence of Theorem \ref{Lipmain} we obtain that the free boundary is smooth in these regions.

The free boundary is not expected to be smooth at the so-called stagnation points where $Q=2gy$.
At such points, the profile of an irrotational wave ($\gamma \equiv 0$)
has a corner with included angle of $120^\circ$. This was
conjectured by Stokes and it was proved by Amick, Fraenkel, and
Toland \cite{AFT}, and by Plotnikov \cite{P}. The case $\gamma\neq
0$ was investigated by Varvaruca in \cite{V} and recently by Varvaruca and Weiss in \cite{VW}.

 The paper is organized as follows. In
Section 2 we introduce notation and definitions and we prove a regularity result for viscosity solutions to a Neumann problem which we will use in the proof of Theorem \ref{flatmain}. Next, in Section 3, we present the statement of our Harnack
inequality and we exhibit its proof. In Section 4, we state and
prove the ``improvement of flatness" lemma. Finally, in Section 5,
we provide the proof of Theorem \ref{flatmain} and Theorem
\ref{Lipmain}. We conclude the paper with an Appendix in which we prove the standard
Lipschitz continuity and non-degeneracy of solutions to a one-phase free boundary problem.

\section{Preliminaries}

In this section we provide notation and definitions used
throughout the paper. We also present an auxiliary result which
will be used in the proof of our main Theorem \ref{flatmain}.

\smallskip

\noindent \textbf{Notation.} For any continuous
function $u: \Omega \subset \R^n \rightarrow \R$ we denote
$$\Omega^+(u):= \{x \in \Omega : u(x)>0\}, \quad F(u):= \partial \Omega^+(u) \cap \Omega.$$
We refer to the set $F(u)$ as to
the free boundary of $u$, while $\Omega^+(u)$ is its
positive phase (or side).

\smallskip

We now state the definition of viscosity solution to the
problem under consideration, that is\begin{equation}\label{fbnew} \left \{
\begin{array}{ll}
    \sum_{i,j} a_{ij}(x) u_{ij} = f,   & \hbox{in $\Omega^+(u)$} \\
\ \\
    |\nabla u|= g, & \hbox{on $F(u).$} \\
\end{array}\right.
\end{equation} Here $\Omega$ is a bounded
domain in $\R^n,$ $a_{ij} \in C^{0,\beta}(\Omega), f \in C(\Omega) \cap L^\infty(\Omega)$, $g \in
C^{0,\beta}(\Omega)$, and $g \geq 0.$

First we need the following standard notion.

\begin{defn}Given $u, \varphi \in C(\Omega)$, we say that $\varphi$
touches $u$ by below (resp. above) at $x_0 \in \Omega$ if $u(x_0)=
\varphi(x_0),$ and
$$u(x) \geq \varphi(x) \quad (\text{resp. $u(x) \leq
\varphi(x)$}) \quad \text{in a neighborhood $O$ of $x_0$.}$$ If
this inequality is strict in $O \setminus \{x_0\}$, we say that
$\varphi$ touches $u$ strictly by below (resp. above).
\end{defn}

\begin{defn}\label{defnhsol} Let $u$ be a nonnegative continuous function in
$\Omega$. We say that $u$ is a viscosity solution to (\ref{fbnew}) in
$\Omega$, if and only if the following conditions are satisfied:
\begin{enumerate}
\item $ \sum_{i,j} a_{ij}(x) u_{ij} = f$ in $\Omega^+(u)$ in the
viscosity sense, i.e if $\varphi \in C^2(\Omega^+(u))$ touches $u$
by below (resp. above) at $x_0 \in \Omega^+(u)$ then
$$ \sum_{i,j} a_{ij}(x_0) \varphi_{ij}(x_0) \leq  f(x_0) \quad (\text{resp. $ \sum_{i,j} a_{ij}(x_0) \varphi_{ij}(x_0) \geq  f(x_0)$}).$$

\item If $\varphi \in C^2(\Omega)$ and $\varphi^+$ touches $u$ by below (resp.  above) at $x_0 \in F(u)$ and $|\nabla \varphi|(x_0) \neq 0$ then $$|\nabla \varphi|(x_0) \leq g(x_0) \quad (\text{resp. $|\nabla \varphi|(x_0) \geq g(x_0)$}).$$
\end{enumerate}
\end{defn}

Viscosity solutions are introduced so to be able to
use comparison techniques. To this aim, we will need the following
notion of comparison subsolution/supersolution.

\begin{defn}\label{defsub} Let $v \in C^2(\Omega)$.
We say that $v$ is a strict (comparison) subsolution (resp.
supersolution) to (\ref{fbnew}) in $\Omega$, if and only if the
following conditions are satisfied:
\begin{enumerate}
\item $ \sum_{i,j} a_{ij}(x) v_{ij}   > f(x)$ (resp. $< f(x)$) in $\Omega^+(v)$;

\item If $x_0 \in F(v)$, then $$|\nabla v|(x_0) > g(x_0) \quad (\text{resp. $0<|\nabla v|(x_0) < g(x_0)$}).$$

\end{enumerate}
\end{defn}

Notice that, by the implicit function theorem, if $v$ is a strict
subsolution/supersolution then $F(v)$ is a $C^2$ hypersurface.

The following lemma is an immediate consequence of the definitions
above.

\begin{lem}\label{comparison}Let $u,v$ be respectively a solution and a strict
subsolution to $\eqref{fbnew}$ in $\Omega$. If $u \geq v^+$ in
$\Omega$ then $u >v^+$ in $\Omega^+(v) \cup F(v).$
\end{lem}


\noindent \textbf {Notation.} Here and after $B_\rho(x_0) \subset
\R^n$ denotes a ball of radius $\rho$ centered at $x_0$, and
$B_\rho=B_\rho (0)$. A positive constant depending only on the
dimension $n$ is called a universal constant. We often use $c,c_i$
to denote small universal constants, and $C,C_i$ to denote large
universal constants.

\

Our main Theorem \ref{flatmain} will follow from the regularity
properties of solutions to the classical Neumann problem for the
Laplace operator. Precisely, we consider the following boundary
value problem:
\begin{equation}\label{Neumann_p}
  \begin{cases}
    \Delta \tilde u=0 & \text{in $B_\rho \cap \{x_n >0\}$}, \\
\ \\
\tilde u_n=0 & \text{on $B_\rho \cap \{x_n =0\}$}.
  \end{cases}\end{equation}

We use the notion of viscosity solution to \eqref{Neumann_p}. For
completeness (and for lack of references), we recall standard
notions and we prove regularity of viscosity solutions.

\begin{defn}\label{defN} Let $\tilde u$ be a continuous function on $B_\rho \cap \{x_n \geq 0\}.$
We say that $\tilde u$ is a viscosity solution to
\eqref{Neumann_p} if given $P(x)$ a quadratic polynomial touching
$\tilde u$ by below (resp. above) at $\bar x \in B_\rho \cap \{x_n
\geq 0\}$, then

\

(i) if $\bar x \in B_\rho \cap \{x_n >0\}$ then $\Delta P \leq 0,$
(resp. $\Delta P \geq 0$) i.e $\tilde u$ is harmonic in the
viscosity sense;

\

(ii) if $\bar x \in B_\rho \cap \{x_n=0\}$ then $P_n(\bar x) \leq
0$ (resp. $P_n(\bar x) \geq 0$.)

\end{defn}

\textbf{Remark.} Notice that, in the definition above we can
choose polynomials $P$ that touch $\tilde u$ strictly by
below/above (replace $P$ by $P_\eta(x) = P(x) - \eta(x_n - \bar
x_n )^2$ and then let $\eta$ go to 0).

Also, it suffices to verify that (ii) holds for polynomials
$\tilde P$ with $\Delta \tilde P
> 0$. Indeed, let $P$ touch $\tilde u$ by below at $\bar x$. Then,
$$\tilde P = P - \eta(x_n-\bar x_n) + C(\eta)(x_n - \bar x_n)^2$$ touches
$\tilde u$ by below at $\bar x$ (for a sufficiently small constant $\eta>0$ and a
large constant $C>0$ depending on $\eta$) and satisfies
$$\Delta \tilde P
>0, \quad \tilde P_n(\bar x) = P_n(\bar x)-\eta.
$$
If (ii) holds for strictly subharmonic polynomials, we get $\tilde
P_n(\bar x) \leq \eta$ which by letting $\eta$ go to 0 implies $P_n(\bar
x) \leq 0$.

\begin{lem}\label{smooth}Let $\tilde u$ be a viscosity solution to
\eqref{Neumann_p}. Then $\tilde u$ is a classical solution to
\eqref{Neumann_p}. In particular,  $\tilde u \in C^\infty(B_\rho \cap
\{x_n \geq 0\}).$
\end{lem}

\begin{proof}Let $$u^*(x)=
  \begin{cases}
    \tilde u(x) & \text{if $x \in B_\rho \cap \{x_n \geq 0\}$}, \\
    \tilde u(x', -x_n) & \text{if $x \in B_\rho \cap \{x_n < 0\}$},
  \end{cases}
$$ where $x'=(x_1,\ldots,x_{n-1}).$

We claim that $u^*$ is harmonic (in the viscosity sense), and
hence smooth, in $B_\rho.$ Indeed, let $P$ be a polynomial
touching $u^*$ at $\bar x \in B_\rho$ strictly by below. We need
to show that $\Delta P \leq 0$. Clearly, we only need to consider
the case when $\bar x \in \{x_n=0\}.$

Consider the polynomial $$S(x) =\frac{P(x)+P(x',-x_n)}{2}.$$ Then
\be\label{S} \Delta S=\Delta P, \quad S_n(x',0)=0.  \ee Also, $S$
still touches $u^*$ strictly by below at $\bar x.$ Now, consider
the family of polynomials $$S_\ep= S+ \ep x_n, \ep >0.$$ For $\ep$
small $S_\ep$ will touch $u^*$ by below at some point $x_\ep$.

If $x_\ep$ belongs to $\{x_n =0\}$, since $S_\ep$ touches $\tilde
u$ by below at $x_\ep$ and $\tilde u_n(x',0)=0$ in the viscosity sense,  we obtain that
$$(S_\ep)_n (x'_\ep,0) \leq 0$$ i.e. $$S_n(x'_{\ep},0) + \ep \leq 0 $$
contradicting \eqref{S}.

Thus $x_\ep \in B_\rho \setminus \{x_n=0\}$ and hence $\Delta
S=\Delta P \leq 0$.

In conclusion, $u^*$ is harmonic in $B_\rho$ and our statement
immediately follows.
\end{proof}

\section{A Harnack inequality}

In this section we will prove a Harnack type inequality for a
solution $u$ to our problem 
\begin{equation}\label{fb3} \left \{
\begin{array}{ll}
    \sum_{i,j} a_{ij}(x) u_{ij} = f,   & \hbox{in $\Omega^+(u):= \{x \in \Omega : u(x)>0\}$,} \\
\ \\
    |\nabla u|= g, & \hbox{on $F(u):= \partial \Omega^+(u) \cap \Omega,$} \\
\end{array}\right.
\end{equation} under the assumption ($0<\ep<1$)
\begin{equation}\label{fb_omega2}
    \|f\|_{L^\infty(\Omega)} \leq \ep^2,  \quad
    \|g(x) - 1\|_{L^\infty(\Omega)} \leq  \ep^2, \quad \|a_{ij}-\delta_{ij}\|_{L^\infty(\Omega)} \leq \ep.\end{equation}

This theorem roughly says that
if the graph of $u$ oscillates $\ep r$ away from $x_n^+$ in $B_r$,
then it oscillates $(1-c)\ep r$ in $B_{r/20}$. A corollary of this
theorem will be a key tool in the proof of Theorem \ref{flatmain}.

\begin{thm}[Harnack inequality]\label{HI}There exists a universal constant $\bar
\ep$,  such that if $u$ solves \eqref{fb3}-\eqref{fb_omega2} and it
satisfies at some point $x_0 \in \Omega^+(u) \cup F(u),$

\be\label{osc} (x_n+ a_0)^+ \leq u(x) \leq (x_n+ b_0)^+ \quad
\text{in $B_r(x_0) \subset \Omega,$}\ee
with
$$b_0 - a_0 \leq \ep r, \quad \text{ $\ep \leq \bar \ep$}$$ then
$$ (x_n+ a_1)^+ \leq u(x) \leq (x_n+ b_1)^+ \quad \text{in
$B_{r/20}(x_0)$},$$ with
$$a_0 \leq a_1 \leq b_1 \leq b_0, \quad
b_1 -  a_1\leq (1-c)\ep r, $$ and $0<c<1$ universal.
\end{thm}

From this statement we immediately get the desired corollary to be
used in the proof of our main result. Precisely, if $u$ satisfies
\eqref{osc} with $r=1$, then we can apply Harnack inequality
repeatedly and obtain $$\label{osc2} (x_n+ a_m)^+ \leq u(x)
\leq (x_n+ b_m)^+ \quad \text{in $B_{20^{-m}}(x_0)$}, $$with
$$b_m-a_m \leq (1-c)^m\ep$$ for all $m$'s such that $$(1-c)^m 20^{m}\ep \leq \bar
\ep.$$ This implies that for all such $m$'s, the oscillation of
the function
$$\tilde u_\ep(x) = \frac{u(x) - x_n }{\ep}$$ in
$(\Omega^+(u) \cup F(u)) \cap B_{r}(x_0), r=20^{-m}$ is less than $(1-c)^m= 20^{-\gamma m} =
r^\gamma$. Thus, the following corollary holds.

\begin{cor} \label{corollary}Let $u$ be a solution to \eqref{fb3}-\eqref{fb_omega2}
satisfying \eqref{osc} for $r=1$. Then  in $B_1(x_0)$ $\tilde
u_\ep$ has a H\"older modulus of continuity at $x_0$, outside
the ball of radius $\ep/\bar \ep,$ i.e for all $x \in (\Omega^+(u) \cup F(u)) \cap B_1(x_0)$, with $|x-x_0| \geq \ep/\bar\ep$
$$|\tilde u_\ep(x) - \tilde u_\ep (x_0)| \leq C |x-x_0|^\gamma.
$$
\end{cor}

The proof of the Harnack inequality relies on the following lemma.

\begin{lem}\label{main}There exists
a universal constant $\bar \ep>0$ such that if $u$ is a solution
to \eqref{fb3}-\eqref{fb_omega2} in $B_1$ with  $0< \ep \leq \bar \ep$
and $u$ satisfies \be\label{control} p(x)^+ \leq u(x) \leq (p(x) +\ep)^+
\quad \text{$x \in B_1$, $p(x)=x_n + \sigma,$ $|\sigma| <
1/10$}\ee then if at $\bar x=\dfrac{1}{5}e_n$ \be\label{u-p>ep2}
u(\bar x) \geq (p(\bar x)+\frac{\ep}{2})^+, \ee then \be u \geq
(p+c\ep)^+ \quad \text{in $\overline{B}_{1/2},$}\ee for some
$0<c<1.$ Analogously, if $$ u(\bar x) \leq (p(\bar x
)+\frac{\ep}{2})^+,$$ then $$ u \leq (p+(1-c)\ep)^+ \quad
\text{in $\overline{B}_{1/2}.$}$$
\end{lem}

\begin{proof} We prove the first statement.
Clearly, from \eqref{control} \be\label{fullcontrol}u \geq p \quad
\text{in $B_1.$}\ee

Let $$w=c(|x-\bar x|^{-\gamma} - (3/4)^{-\gamma})$$ be defined in the closure of the annulus
$$A:=  B_{3/4}(\bar x) \setminus \overline{B}_{1/20}(\bar x).$$
The constant $c$ is such that
$w$ satisfies the
boundary conditions
  $$\begin{cases}
    w =0 & \text{on $\p B_{3/4}(\bar x)$}, \\
    w=1 & \text{on $\p B_{1/20}(\bar x)$}.
  \end{cases} $$ Also, since $\|a_{ij}-\delta_{ij}\|_{L^\infty(B_1)} \leq \ep$ the matrix $A=a_{ij}$ is uniformly elliptic and we can choose the constant $\gamma$ universal so that
$$\sum_{ij} a_{ij}(x)w_{ij} \geq \delta>0 \quad \text{in $A$},$$ with $\delta$ universal.
Extend $w$ to be equal
to 1 on $B_{1/20}(\bar x).$

Notice that since $|\sigma| < 1/10$ using \eqref{fullcontrol} we get
\be\label{inclusion1} B_{1/10}(\bar x) \subset B_1^+(u). \ee Also,
$$ B_{1/2} \subset \subset B_{3/4}(\bar x) \subset \subset B_1. $$

Since in view of \eqref{fullcontrol}-\eqref{inclusion1}, $u-p \geq
0$ and solves a uniformly elliptic equation in $B_{1/10}(\bar x)$ with right-hand side $f$, we can apply Harnack inequality
to obtain
\be\label{HInew} u(x)
- p(x) \geq c(u(\bar x)- p(\bar x)) - C \|f\|_{L^\infty} \quad \text{in $\overline B_{1/20}(\bar x)$}. \ee
From \eqref{u-p>ep2} and the first inequality in \eqref{fb_omega2} we conclude that (for $\ep$ small enough)
\be\label{u-p>cep} u
- p \geq c\ep - C \ep^2 \geq c_0\ep \quad \text{in $\overline B_{1/20}(\bar x)$}. \ee
Now set

\be\label{v} v(x)= p(x)+ c_0\ep (w(x)-1), \quad x \in \overline B_{3/4}(\bar x),\ee and for $t
\geq 0,$
$$v_t(x)= v(x)+t, \quad x \in \overline B_{3/4}(\bar x).
$$ Notice that,
$$\sum_{ij} a_{ij}(x) (v_t)_{ij} \geq c_0\delta \ep > \ep^2 \quad \text{in $A$.}$$
According to \eqref{fullcontrol} and the definition of $v_t$ we have,$$\label{v<u} v_0(x)=v(x) \leq
p(x) \leq u(x) \quad x \in \overline B_{3/4}(\bar x).$$

Let $\bar t$ be the largest $t \geq 0$ such that
$$v_{t}(x) \leq u(x) \quad \text{in $\overline B_{3/4}(\bar x)$}.$$

We want to show that $\bar t \geq c_0\ep.$ Then, using the
definition \eqref{v} of $v(x)$  we get
$$u(x) \geq v(x) + \bar t \geq p(x) + c_0\ep w(x)$$
and hence, since  on $\overline B_{1/2} \subset B_{3/4}(\bar x)$
one has $w(x) \geq c_2$ for some universal constant $c_2$, we
obtain that $$ u(x)-p(x) \geq c \ep  \quad
\text{on $\overline B_{1/2}$}$$ as desired.

Suppose $\bar t < c_0\ep$. Then at some
$\tilde x \in \overline B_{3/4}(\bar x)$ we have $$v_{\bar t}(\tilde x) =
u(\tilde x).$$ We show that such touching point can only occur on $\overline B_{1/20}(\bar x).$
Indeed, since $w\equiv 0$ on $\p B_{3/4}(\bar x)$ from the definition of $v_t$ we get

$$v_{\bar t}(x) = p(x) -c_0\ep +\bar t \quad \textrm{on  $\p B_{3/4}(\bar x)$}.$$

\noindent Using that $\bar t< c_0\ep$ together with the fact that $u \geq p$ we then obtain

$$v_{\bar t} < u \quad \textrm{on  $\p B_{3/4}(\bar x)$}.$$

We now show that $\tilde x$ cannot belong to the annulus $A$.
As already observed, $$\label{laplacev} \sum_{ij} a_{ij}(x)(v_{\bar t})_{ij} > \ep^2, \quad \textrm{in
$A$}$$ and also
\be\label{boundongradv} |\nabla v_{\bar t}| \geq |v_n| = |1+
c_0\ep w_n|, \quad \textrm{in $A$}.\ee


We claim that $$w_n(x) \geq c_1 \quad \text{on $\{v_{\bar t} \leq
0\} \cap A$},
$$ for a universal constant $c_1.$

Indeed, since $w$ is radially symmetric, $$ \label{w_n}w_n(x)=
|\nabla w(x)|\nu_x \cdot e_n, \quad \text{$x \in A$ }$$ where
$\nu_x$ is the unit direction of $x-\bar x$. Clearly from the formula for $w$ we get that $|\nabla w|>c$ on $A.$ Also, $\nu_x \cdot
e_n$ is bounded below in the region $\{v_{\bar t} \leq 0\} \cap
A,$ since for $\ep$ small enough $$ \{v_{\bar t} \leq 0\} \cap A
\subset \{p \leq c_0\ep\}= \{x_n \leq -\sigma + c_0\ep\} \subset \{x_n
<3/20\},
$$ and $\bar x = 1/5e_n.$


Hence, from \eqref{boundongradv} we deduce that $$ |\nabla v_{\bar t}|
\geq 1 + c_2 \ep, \quad \text{on $\{v_{\bar t}\leq 0\}\cap A.$}$$
In particular, for $\ep$ small enough and in view of the second inequality in \eqref{fb_omega2},
$$|\nabla v_{\bar t}|(x)  >
1+ \ep^2 \geq g(x) \quad \text{for  $x \in A \cap F(v_{\bar t})$}.$$ Thus,
$v_{\bar t}$ is a strict subsolution to $\eqref{fb3}$ in $A$
and according to Lemma \ref{comparison} since $u$ solves
\eqref{fb3} in $B_1$, $\tilde x$ cannot belong to $A.$
Therefore, $\tilde x \in \overline B_{1/20}(\bar x)$ and
$$u(\tilde x)=v_{\bar t}(\tilde x) \leq p(\tilde x)+\bar t < p(\tilde
x)+c_0\ep,$$ which implies
$$u(\tilde x) - p(\tilde x) < c_0\ep$$ contradicting \eqref{u-p>cep}.

The proof of the second statement follows from a similar argument.
\end{proof}

We are now ready to give the proof of the Harnack inequality.

\

\textit{Proof of Theorem \ref{HI}.} Assume without loss of
generality,
$$x_0=0, \quad r=1.$$
 According to \eqref{osc}, $$\label{osc1} p(x)^+ \leq u(x) \leq
(p(x)+\ep)^+, \quad \text{in $B_1,$}$$ with $p(x)=x_n+ a_0$. If
$|a_0| < 1/10$ then we can apply the previous Lemma \ref{main} and the
desired statement immediately follows.

Suppose not. If $a_0 < -1/10$, then (for $\ep$ small) 0 belongs to the zero phase of
$(p(x)+\ep)^+$ which implies that 0 also belongs to the zero phase
of $u$, a contradiction.

If $a_0 >1/10$ then $B_{1/10} \subset B^+_1(u)$, and the
conclusion follows by the classical Harnack inequality
in $B_{1/10}$ as long as $\ep$ is small enough.\qed

\section{Improvement of flatness}


In this section we present the main ``improvement of flatness" lemma, from which the proof of
Theorem \ref{flatmain} will easily follow via an iterative
argument.

\begin{lem}[Improvement of flatness] \label{improv}Let $u$ be a
solution to \eqref{fb3}-\eqref{fb_omega2} in $B_1$ satisfying
\begin{equation}\label{flat}(x_n -\ep)^+ \leq u(x) \leq (x_n +
\ep)^+ \quad \text{for $x \in B_1,$}
\end{equation} with $ 0 \in F(u).$

If $0<r \leq r_0$ for $r_0$ a
universal constant and $0<\ep \leq \ep_0$ for some $\ep_0$
depending on $r$, then

\begin{equation}\label{improvedflat_2}(x \cdot \nu -r\frac{\ep}{2})^+  \leq u(x) \leq
(x \cdot \nu +r\frac{\ep }{2})^+ \quad \text{for $x \in B_r,$}
\end{equation}with $|\nu|=1,$ and $ |\nu - e_n| \leq C\ep^2$ for a
universal constant $C.$

\end{lem}

\begin{proof}We divide the proof of this Lemma into 3 steps. We
use the following notation:
$$\Omega_\rho(u):= (B_1^+(u) \cup F(u)) \cap B_\rho.$$

 \

\textbf{Step 1 -- Compactness.} Fix $r \leq r_0$ with $r_0$ universal (the precise $r_0$ will be given in Step 3). Assume by contradiction that we
can find a sequence $\ep_k \rightarrow 0$ and a sequence $u_k$ of
solutions to \eqref{fb3} in $B_1$ with coefficients $a_{ij}^k$, right hand side $f_k$ and free boundary condition $g_k$ satisfying \eqref{fb_omega2}, such that $u_k$ satisfies \eqref{flat},
i.e.
\begin{equation}\label{flat_k}(x_n -\ep_k)^+ \leq u_k(x) \leq (x_n +
\ep_k)^+ \quad \text{for $x \in B_1$,  $0 \in F(u_k),$}
\end{equation}
but it does not satisfy the conclusion \eqref{improvedflat_2} of the lemma.

Set,$$ \tilde{u}_{k}(x)= \frac{u_k(x) - x_n}{\ep_k}, \quad x \in
\Omega_1(u_k).
$$
Then \eqref{flat_k} gives,
\begin{equation}\label{flat_tilde} -1 \leq \tilde{u}_{k}(x) \leq 1
\quad \text{for $x \in \Omega_1(u_k)$}.
\end{equation}

From Corollary \ref{corollary}, it follows that the function
$\tilde u_{k}$ satisfies \be\label{HC}|\tilde u_{k}(x) - \tilde
u_{k} (y)| \leq C |x-y|^\gamma,\ee for $C$ universal and
$$|x-y| \geq \ep_k/\bar\ep, \quad x,y \in \Omega_{1/2}(u_k).$$ From \eqref{flat_k} it clearly follows that
$F(u_k)$ converges to $B_1 \cap \{x_n=0\}$ in the Hausdorff
distance. This fact and \eqref{HC} together with Ascoli-Arzela
give that as $\ep_k \rightarrow 0$ the graphs of the
$\tilde{u}_{k}$ over $\Omega_{1/2}(u_k)$ converge (up to a
subsequence) in the Hausdorff distance to the graph of a H\"older
continuous function $\tilde{u}$ over $B_{1/2} \cap \{x_n \geq
0\}$.

\

\textbf{Step 2 -- Limiting Solution.} We now show that $\tilde u$
solves
\begin{equation}\label{Neumann}
  \begin{cases}
    \Delta \tilde u=0 & \text{in $B_{1/2} \cap \{x_n >0\}$}, \\
\ \\
\tilde u_n=0 & \text{on $B_{1/2} \cap \{x_n =0\}$},
  \end{cases}\end{equation}
in the sense of Definition \ref{defN}.



Let $P(x)$ be a quadratic polynomial touching $\tilde u$ at $\bar
x \in B_{1/2} \cap \{x_n \geq 0\}$ strictly by below. We need to
show that

\

(i) if $\bar x \in B_{1/2} \cap \{x_n >0\}$ then $\Delta P \leq
0;$

\

(ii) if $\bar x \in B_{1/2} \cap \{x_n=0\}$ then $P_n(\bar x) \leq
0.$

\

Since $\tilde u_{k} \rightarrow \tilde{u}$ in the sense specified
above, there exist points $x_k \in \Omega_{1/2}(u_k)$, $x_k
\rightarrow \bar x$, and constants $c_k \rightarrow 0$ such that
\be\label{P+=u} P(x_k)+c_k = \tilde u_{k}(x_k)\ee and
\be\label{uaboveP+} \tilde u_{k} \geq P+c_k \quad \text{in a
neighborhood of $x_k$}. \ee


From the definition of $\tilde u _{k}$, $\eqref{P+=u}$ and
$\eqref{uaboveP+}$ read as $$\label{u=Q} u_k(x_k)= Q(x_k)$$ and
$$\label{u>Q} u_k(x) \geq Q(x) \quad \text{in a neighborhood of
$x_k$}$$ where $$Q(x)= \ep_k(P(x)+c_k) + x_n.$$

We now distinguish the two cases.

\

 (i) If $\bar x \in B_{1/2} \cap
\{x_n
>0\}$ then $x_k \in B_{1/2}^+(u_k)$ (for $k$ large) and hence since $Q$ touches $u_k$ by below at $x_k$ we get
$$\sum_{i,j}a^k_{ij}(x_k) Q_{ij}= \ep_k\sum_{i,j}a^k_{ij}(x_k) P_{ij} \leq f_k(x_k)\le \ep_k^2.$$ Thus, in view of the last inequality in \eqref{fb_omega2}
$$\Delta P = \sum_{i,j}(\delta_{ij}-a^k_{ij}(x_k))P_{ij} + \sum_{i,j}a^k_{ij}(x_k)P_{ij} \leq C\ep_k.$$ Passing to the limit as $k \rightarrow +\infty$
we obtain that $\Delta P \leq 0$ as desired.

\smallskip

(ii) If $\bar x \in B_{1/2} \cap \{x_n=0\}$, as observed in the
Remark following Definition \ref{defN}, we can assume that $\Delta
P
>0.$ We claim that
for $k$ large enough, $x_k \in F(u_k)$. Otherwise $x_{k_n} \in
B_1^+(u_{k_n})$ for a subsequence $k_n \rightarrow \infty$ and as
in the case (i)$$ \Delta P \leq C\ep_{k_n}.$$ Letting $k_n
\rightarrow \infty $ we contradict the fact that $P$ is strictly
subharmonic. Thus $x_k \in F(u_k)$ for $k$ large. Now notice that
$$\nabla Q = \ep_k \nabla P + e_n$$ thus, for $k$ large, $|\nabla
Q|>0$. Since $Q^+$ touches
$u_k$ by below,
$$|\nabla Q|(x_k) \leq g_k(x_k) \leq 1 + \ep_k^2,$$ which gives, $$|\nabla Q|^2(x_k)=\ep_k^2
|\nabla P|^2(x_k) + 1 + 2\ep_k P_n(x_k) \leq 1 +3\ep_k^2,$$ and
thus (after division by $\ep_k$)
$$\ep_k |\nabla P|^2(x_k) - 3\ep_k + 2 P_n(x_k) \leq 0.$$
Passing to the limit as $k \rightarrow +\infty$ we obtain
$$ P_n(\bar x) \leq 0$$ as desired.

\

\textbf{Step 3 -- Improvement of flatness.} From the previous
step, $\tilde u$ solves \eqref{Neumann} and from
\eqref{flat_tilde},
$$-1 \leq \tilde u \leq 1 \quad \text{in $B_{1/2} \cap \{x_n \geq
0\}.$}
$$

From Lemma \ref{smooth} and the bound above we obtain that, for
the given $r$,
$$|\tilde u(x) - \tilde u(0) - \nabla \tilde u(0)\cdot x| \le C_0r^2 \quad \text{in $B_r \cap \{x_n \geq 0\}
$},$$ for a universal constant $C_0$.  In particular, since $0 \in
F(\tilde u)$ and also $\tilde u_n (0) = 0$, we obtain
$$ x' \cdot \tilde{\nu} - C_0 r^2 \leq \tilde{u}(x) \leq x' \cdot
\tilde {\nu}+C_0r^2 \quad \text{in $B_r \cap \{x_n \geq 0\}$},$$
with $\tilde \nu_i= \tilde u_i(0), i=1,\ldots, n-1, |\tilde \nu| \leq
\tilde C$, $\tilde C$ universal constant. Therefore, for $k$ large
enough we get,
$$ x' \cdot \tilde{\nu} - C_1 r^2 \leq \tilde{u}_{k}(x) \leq x'
\cdot \tilde {\nu}+C_1 r^2 \quad \text{in $\Omega_r(u_k)$}.$$ From
the definition of $\tilde{u}_{k}$ the inequality above reads
\be\label{almostflat} \ep_k x' \cdot \tilde{\nu} + x_n - \ep_k C_1
r^2 \leq u_k \leq \ep_k x' \cdot \tilde {\nu}+x_n + \ep_k C_1r^2
\quad \text{in $\Omega_r(u_k)$}.\ee Call
$$\nu= \frac{(\ep_k \tilde \nu, 1)}{\sqrt{\ep_k^2+1}}.
$$ Since, for $k$ large, $$1 \leq \sqrt{\ep_k^2+1} \leq 1 +
\frac{\ep_k^2}{2},$$ we deduce from \eqref{almostflat} that $$ x
\cdot \nu  - \frac{\ep_k^2}{2}r - C_1 r^2\ep_k \leq u_k \leq x
\cdot \nu + \frac{\ep_k^2}{2}r+ C_1r^2 \ep_k \quad \text{in
$\Omega_r(u_k)$}.$$ In particular, if $r_0$ is such that $C_1r_0
\leq 1/4$ and also $k$ is large enough so that $\ep_k \leq 1/2$ we
obtain $$ x \cdot \nu  - \frac{\ep_k}{2}r \leq u_k \leq x \cdot
\nu + \frac{\ep_k}{2}r \quad \text{in $\Omega_r(u_k)$},$$ which
together with \eqref{flat_k} implies that  $$ (x \cdot \nu  -
\frac{\ep_k}{2}r)^+ \leq u_k \leq (x \cdot \nu +
\frac{\ep_k}{2}r)^+ \quad \text{in $B_r$}.$$ Thus the $u_k$
satisfy the conclusion of the lemma, and we reached a
contradiction.
\end{proof}

\section{The proofs of Theorem \ref{flatmain} and Theorem \ref{Lipmain}.}

In this section we finally present the proof of our main theorems.

\

 \textbf{Proof of Theorem
\ref{flatmain}.} Let $u$ be a viscosity solution to \eqref{fb} in
$B_1$, with $0 \in F(u)$, $g(0)=1$ and $a_{ij}(0)= \delta_{ij}$. Consider the sequence of
rescalings
$$u_k(x) := \frac{u(\rho_k x)}{\rho_k}, \quad x\in B_1,$$
with $\rho_k=\bar r^k,$ $k=0,1,\ldots$, for a fixed $\bar r$ such
that
$$\bar r^\beta \leq 1/4, \quad \bar r \leq r_0,$$ with $r_0$ the universal constant in Lemma
\ref{improv}.

Each $u_k$ solves \eqref{fb} in $B_1$ with coefficients $a_{ij}^k(x)= a_{ij}(\rho_k x)$, right hand side $f_k(x):=\rho_kf(\rho_k x),$ and free boundary condition $g_k(x):=g(\rho_k x)$.
For the chosen $\bar r$, by taking $\bar \ep = \ep_0(\bar
r)^2$ the assumption \eqref{fb_omega2} holds for
$\ep=\ep_k :=2^{-k}\ep_0(\bar r)$. Indeed, in $B_1$, in view of \eqref{hypflatmain}, $$|f_k(x)| \leq \|f\|_{L^\infty} \rho_k \leq
\bar \ep \bar r^k\leq \ep_k^2, $$
$$|g_k(x) - 1|=|g(\rho_k x) - g(0)| \leq [g]_{0,\beta} \rho_k^\beta \leq
\bar \ep \bar r^{k\beta}\leq \ep_k^2, $$ and
$$|a^k_{ij}(x) -\delta_{ij}| = |a_{ij}(\rho_k x) - a_{ij}(0)| \leq [a_{ij}]_{0,\beta} \rho_k^\beta \leq
\bar \ep \bar r^{k\beta}\leq \ep_k.$$

The hypothesis \eqref{hyp1} guarantees that for $k=0$ also the
flatness assumption \eqref{flat} in Lemma \ref{improv} is
satisfied by $u_0$. Then, it easily follows by induction on $k$
and Lemma \ref{improv} that each $u_k$ is $\ep_k$-flat in $B_1$ in the sense of
\eqref{flat}. Now, a standard iteration argument gives the desired
statement.\qed

\

\textbf{Proof of Theorem \ref{Lipmain}.} Let $u$ be a viscosity
solution to \eqref{fb}, with $0 \in F(u)$ and $g(0)>1.$ Without loss
of generality, assume $g(0)=1.$ Also, for simplicity we take $a_{ij}(0) = \delta_{ij}.$

Consider the blow-up sequence
$$u_k:= u_{\delta_k}(x) = \frac{u(\delta_k x)}{\delta_k}, $$
with $\delta_k \rightarrow 0$ as $k \rightarrow \infty.$ As in the previous theorem, each $u_k$ solves \eqref{fb} with coefficients $a_{ij}^k(x)= a_{ij}(\delta_k x)$, right hand side $f_k(x):=\delta_kf(\delta_k x),$ and free boundary condition $g_k(x):=g(\delta_k x)$. For $k$ large, the assumption \eqref{hypflatmain} is satisfied for the universal constant $\bar \ep$. In fact, in $B_1$
$$|f_k(x)|= \delta_k|f(\delta_k x)| \leq \delta_k \|f\|_{L^\infty} \leq \bar \ep  $$
$$|g_k(x)-1|= |g_k(x)-g(0)| \leq [g_k]_{0,\beta} = \delta_k^\beta[g]_{0,\beta} \leq \bar \ep, $$
and
$$|a_{ij}^k(x)-\delta_{ij}|=|a_{ij}(\delta_k x)-a_{ij}(0)| \leq [a_{ij}(\delta_k x)]_{0,\beta} = \delta_k^{\beta}[a_{ij}]_{0,\beta} \leq
\bar \ep.$$ Thus, using
non-degeneracy and uniform Lipschitz continuity of the
$u_k$'s (see Appendix for a proof of these properties), standard arguments (see for example \cite{AC}) give that (up
to extracting a subsequence):
\begin{enumerate}
\item $u_k \rightarrow u_0 \ \ \text{in $C^{0,\alpha}_{loc}(\R^n)$}, \ \ \text{for all $0 < \alpha
<1$};$
\item $\p \{u_k >0 \} \rightarrow \p \{u_0 >0 \}$ locally in the Hausdorff
distance;\\
\end{enumerate}
for a globally defined function $u_0: \R^n \longrightarrow \R$.
The blow-up limit $u_0$ is a global solution to the free boundary
problem
\begin{equation}\label{fb_h} \left \{
\begin{array}{ll}
    \Delta u_0 = 0,   & \hbox{in $\{u_0>0\}$,} \\
\ \\
    |\nabla u_0|= 1, & \hbox{on $F(u_0)$,} \\
\end{array}\right.
\end{equation} and since $F(u)$ is a Lipschitz graph in a neighborhood of $0$ we also have from (i)-(ii) that $F(u_0)$ is
Lipschitz continuous. Thus, it follows from \cite{C1} that $u_0$
is a so-called one-plane solution, i.e. (up to rotations) $u_0 =
x_n^+.$ Combining the facts above, one concludes that for all $k$
large enough, $u_k$ is $\bar \ep$-flat say in $B_1 $ i.e. $$\label{hyp1new}(x_n-\bar \ep)^+ \leq u_k(x) \leq
(x_n+\bar \ep)^+, \quad x \in B_1.$$
Thus $u_k$ satisfies the assumptions of Theorem \ref{flatmain},
and our conclusion follows. \qed

\section{Appendix}

We sketch here the proof of a standard result that is Lipschitz continuity and non-degeneracy of a solution $u$ to
\begin{equation}\label{fb4} \left \{
\begin{array}{ll}
    \sum_{i,j} a_{ij}(x) u_{ij} = f,   & \hbox{in $\Omega^+(u):= \{x \in \Omega : u(x)>0\}$,} \\
\ \\
    |\nabla u|= g, & \hbox{on $F(u):= \partial \Omega^+(u) \cap \Omega,$} \\
\end{array}\right.
\end{equation} under the assumption ($0<\ep<1$)
\begin{equation}\label{fb5}
    \|f\|_{L^\infty(\Omega)} \leq \ep^2,  \quad
    \|g(x) - 1\|_{L^\infty(\Omega)} \leq  \ep^2, \quad \|a_{ij}-\delta_{ij}\|_{L^\infty(\Omega)} \leq \ep.\end{equation}

\begin{lem}Let $u$ be a solution to \eqref{fb4}-\eqref{fb5} with $\ep \leq \tilde \ep$ a universal constant. If $F(u) \cap B_1 \neq \emptyset$ and $F(u)$ is a Lipschitz graph in
$B_2$, then u is Lipschitz and non-degenerate in $B_1^+(u) $ i.e.
\begin{equation}\label{nondegeq}
c_0 d(z) \leq u(z) \leq C_0 d(z) \quad \textrm{for all $z \in B_1^+(u) $,}
\end{equation}
with $d(z) =\textrm{dist}(z, F(u))$ and $c_0,C_0$ universal constants.
\end{lem}

\begin{proof} Assume without loss of generality that $0 \in B_1^+(u)$ and call $d:=d(0)$.

Consider the rescaled function $$\tilde{u}(x)= \frac{u(dx)}{d}, \quad x \in B_1.$$ Clearly $\tilde u$ still satisfies \eqref{fb} in $B_1$ with coefficients $\tilde{a}_{ij}(x)=a_{ij}(dx),$ right hand side $\tilde f (x)= d f(dx)$ and free boundary condition $\tilde g (x)= g(dx)$. Since $d \leq 1$, the assumption \eqref{fb_omega2} holds. We wish to show that $$c_0\leq \tilde u (0) \leq C_0.$$
Assume by contradiction that $\tilde u(0) > C_0$, with $C_0$ to be made precise later.

To construct a subsolution, we use the same function as in Lemma \ref{main}. Precisely, let $$G(x)= C(|x|^{-\gamma}-1)$$ be defined on the closure of the annulus $B_1\setminus  \overline{B}_{1/2}$. In view of the uniform ellipticity of the coefficients, we can choose $\gamma$ large universal so that (for $\ep$ small)
$$\sum_{ij}\tilde a_{ij}G_{ij}>\ep^2 \quad \textrm{on  $B_1\setminus  \overline{B}_{1/2}$}$$  and we can choose $C$ so that $$G= 1 \quad \textrm{on $\partial B_{1/2}.$}$$
By Harnack inequality (see \eqref{HInew}), using the contradiction hypothesis we get (for $\ep$ small) $$\tilde u \geq c \tilde u(0) \quad \textrm{on $\overline{B}_{1/2}$.}$$ Thus, by the maximum principle $$\tilde u(x) \geq v(x) = c \tilde u(0) G(x) \quad \textrm{on $\overline{B}_1 \setminus B_{1/2}$.}$$  Hence at the point $z$ where $d(0)$ is achieved we have $$|\nabla v|(z) \leq g(z) \leq 1+\ep^2 \leq 2$$ which contradicts $\tilde u(0) > C_0$ if $C_0$ is large enough.

To prove the lower bound, let $$\tilde G(x) = \eta (1-G(x))$$ with $\eta$ (depending on $\gamma$) such that $$|\nabla \tilde G| < 1-\ep^2 \quad \textrm{on $\p B_{1/2}.$}$$

Assume without loss of generality that $F(u)$ is a Lipschitz graph in the $x_n$ direction with Lipschitz constant equal to 1. We translate the graph of $\tilde{G}$ by $-4e_n$. Notice that it is above the graph of $\tilde u$  since $\tilde u \equiv 0$ in $B_1(-4e_n).$ We slide the graph of $\tilde G$ in the $e_n$ direction till we touch the graph of $\tilde u$. Since $\tilde G$ is a strict supersolution to our free boundary problem, the touching point $\tilde z$ can occur only on the $\eta $ level set with $\tilde d := d(\tilde z, F(u)) \leq 1.$
From the first part, $\tilde u$ is  Lipschitz continuous and hence $\tilde u (\tilde z) = \eta \leq C \tilde d$. Thus $$C^{-1}\eta \leq \tilde d\leq 1$$ that is $\tilde d$ is comparable to 1. Since $F(u)$ is Lipschitz we can connect $0$ and $\tilde z$ with a chain of intersecting balls included in the positive side of $\tilde u$ with radii comparable to 1.  The number of balls is bounded by a universal constant . Then we can apply Harnack inequality and obtain (for $\ep$ small) $$\tilde u(0) \geq c\tilde u(\tilde z)= c_0,$$ as desired.\end{proof}


\begin{thebibliography}{9999}\bibitem[AC]{AC} Alt H.W.,
Caffarelli L.A., \emph{Existence and regularity for a minimum problem
with free boundary}, J. Reine Angew. Math {\bf 325}
(1981),105--144.
\bibitem[AFT]{AFT} Amick C.J., L. E. Fraenkel L.E., Toland J.F., {\it On the Stokes
conjecture for the wave of extreme form}, Acta Math., 148 (1982),
193-214.
\bibitem[AF]{AF} Argiolas R., Ferrari F., {\it Falt free boundaries regularity in two-phase problems for a class of fully nonlinear elliptic operators with variable coefficients,} Interfaces Free Bound. 11 (2009), no.2, 177-199.

\bibitem[C1]{C1} Caffarelli L.A., \emph{A Harnack
inequality approach to the regularity of free boundaries. Part I:
Lipschitz free boundaries are $C^{1,\alpha}$}, Rev. Mat.
Iberoamericana \textbf{3} (1987) no. 2, 139--162.
\bibitem[C2]{C2} Caffarelli L.A., \emph{A Harnack
inequality approach to the regularity of free boundaries. Part II:
Flat free boundaries are Lipschitz}, Comm. Pure Appl. Math.
\textbf{42} (1989), no.1, 55--78.
\bibitem[CC]{CC} Caffarelli L.A., Cabre X., {\it Fully Nonlinear Elliptic
Equations}, Colloquium Publications 43, American Mathematical
Society, Providence, RI, 1995.
\bibitem[CFS]{CFS} Cerutti M.C., Ferrari F., Salsa S., {\it Two phase
problems for linear elliptic operators with variable coefficients:
Lipschitz free boundaries are $C^{1,\gamma}$}, Archive for
Rational Mechanics and Analysis, Vol 171, n.3, pp. 329 - 348
(2004)
\bibitem[CS]{CS} Constantin A., Strauss W., {\it Exact steady periodic water waves
with vorticity}, Comm. Pure Appl. Math., 57 (2004), 481-527.

\bibitem[F1]{F1} Feldman M., {\it Regularity of Lipschitz free boundaries in two-phase problems for fully nonlinear elliptic equations}, Indiana Univ. Math. J. 50 (2001), no.3, 1171--1200
\bibitem[F2]{F2} Feldman M., {\it Regularity for nonisotropic two-phase problems with Lipshitz free boundaries}, Differential Integral Equations 10 (1997), no.6, 1171--1179.
\bibitem[Fe1]{Fe1} Ferrari F., {\it Two-phase problems for a class of fully nonlinear elliptic operators, Lipschitz free boundaries are $C^{1,\gamma}$},
Amer. J. Math. 128 (2006), 541--571.


\bibitem[FS1]{FS1} Ferrari F., Salsa S., {\it Regularity of the free boundary in two-phase problems for elliptic operators},
 Adv. Math. 214 (2007), 288--322.


\bibitem[FS2]{FS2} Ferrari F., Salsa S., {\it Subsolutions of elliptic operators in divergence form and application to two-phase free boundary problems}, Bound. Value Probl. 2007, art. ID 57049, 21pp.


\bibitem[KN]{KN} Kinderlehrer D., Nirenberg L., {\it Analyticity at the
boundary of solutions of nonlinear second-order parabolic
equations,}  Comm. Pure Appl. Math.  31  (1978), no. 3, 283--338.
\bibitem[P]{P} Plotnikov P.I., {\it Proof of the Stokes conjecture in the theory of
surface waves}, (In Russian), Dinamika Splosh. Sredy, 57 (1982),
41-76. English translation: Stud. Appl. Math., 3 (2002), 217-244.

\bibitem[S]{S} Savin O., {\it Small perturbation solutions for elliptic
equations,} Comm. Partial Differential Equations, {\bf 32},
557--578, 2007.
\bibitem[V]{V} Varvaruca E., {\it On the existence of extreme waves and the Stokes conjecture with
vorticity}, Preprint 2008, arXiv:0707.2224.

\bibitem[VW]{VW}  Varvaruca E., Weiss G.S., {\it A geometric approach to generalized Stokes conjectures}, Preprint 2009, arXiv:0908.1031.


\bibitem[W1]{W1}  Wang P.Y., {\it Regularity of free boundaries of two-phase problems for fully nonlinear elliptic equations of second order. I. Lipschitz free boundaries are $C^{1,\alpha}$}, Comm. Pure Appl. Math. 53 (2000), 799--810.
\bibitem[W2]{W2}  Wang P.Y., {\it Regularity of free boundaries of two-phase problems for fully nonlinear elliptic equations of second order. II. Flat free boundaries are Lipschitz}, Comm. Partial Differential Equations 27 (2002), 1497--1514.

\end{thebibliography}
\end{document}